\documentclass[12pt]{article}

\title{On the indecomposability of $\omega^n$}
\date{May 9, 2011\\{\small(Revised October 18, 2011)}}
\author{Jared R. Corduan \and Fran{\c c}ois G. Dorais}

\usepackage{verbatim}
\usepackage{enumerate}
\usepackage{amsmath}
\usepackage{amssymb}
\usepackage{amsthm}
\usepackage{hyperref}
\usepackage[all]{xy}

\newcommand{\hyphen}{\mbox{-}}

\newcommand{\THEN}{\mathrel{\Rightarrow}}

\newcommand{\IFF}{\mathrel{\Leftrightarrow}}

\newcommand{\lthen}{\mathrel{\rightarrow}}

\newcommand{\liff}{\mathrel{\leftrightarrow}}

\renewcommand{\models}{\mathrel{\vDash}}
\newcommand{\forces}{\mathrel{\Vdash}}

\newcommand{\nforces}{\mathrel{\nVdash}}

\newcommand{\NN}{\mathbb{N}}
\newcommand{\MN}{\mathfrak{N}}
\newcommand{\MM}{\mathfrak{M}}
\newcommand{\FN}{\mathcal{N}}
\newcommand{\PP}{\mathbb{P}}
\newcommand{\QQ}{\mathbb{Q}}

\newcommand{\set}[1]{\lbrace#1\rbrace}
\newcommand{\seq}[1]{\langle#1\rangle}
\newcommand{\cat}{\mathop{{}^\frown}}

\newcommand{\GINO}{\mathsf{G}}
\newcommand{\HUGO}{Player~\ensuremath{\oplus}}
\newcommand{\TINO}{Player~\ensuremath{\ominus}}

\newcommand{\Ind}{\mathsf{I}}
\newcommand{\Bnd}{\mathsf{B}}
\newcommand{\Reg}{\mathsf{R}}

\newcommand{\ACA}{\mathsf{ACA}}
\newcommand{\RCA}{\mathsf{RCA}}

\newcommand{\EIndec}{\mathsf{Elem{\hyphen}Indec}}

\newcommand{\GIndec}{\mathsf{Game{\hyphen}Indec}}
\newcommand{\LIndec}{\mathsf{Lex{\hyphen}Indec}}
\newcommand{\XIndec}{\mathsf{Iso{\hyphen}Indec}}

\newcommand{\st}{\mathsf{S}}
\newcommand{\wk}{\mathsf{W}}
\newcommand{\hw}{\mathsf{HW}}
\newcommand{\RT}{\mathsf{RT}}
\newcommand{\IPT}{\mathsf{IPT}}
\newcommand{\ADS}{\mathsf{ADS}}
\newcommand{\CAC}{\mathsf{CAC}}

\newcounter{statement}[section]

\theoremstyle{plain}
\newtheorem{theorem}[statement]{Theorem}
\newtheorem{proposition}[statement]{Proposition}
\newtheorem{lemma}[statement]{Lemma}
\newtheorem{corollary}[statement]{Corollary}
\theoremstyle{definition}
\newtheorem{question}[statement]{Question}
\newtheorem{definition}[statement]{Definition}

\newenvironment{axiom}[1]{\description\item[{\boldmath#1} ---]\itshape}{\enddescription}
\newenvironment{principle}{\itemize\item[---]}{\enditemize}

\begin{document}

\maketitle

\begin{abstract}\noindent
  We study the reverse mathematics of pigeonhole principles for finite
  powers of the ordinal $\omega$. Four natural formulations are
  presented and their relative strengths are compared. In the analysis
  of the pigeonhole principle for $\omega^2$, we uncover two weak
  variants of Ramsey's Theorem for pairs.
\end{abstract}

\section{Introduction}

In the set-theoretic literature, one finds two formulations of
the indecomposability of an ordinal $\alpha$:
\begin{description}
\item[Additive Indecomposability ---] If $\beta_0 + \cdots +
  \beta_{k-1} = \alpha$ then $\beta_i = \alpha$ for some $i < k$.
\item[Combinatorial Indecomposability ---] If $B_0 \cup \cdots \cup
  B_{k-1} = \alpha$ then $B_i$ has order-type $\alpha$ for some $i <
  k$.
\end{description}
Prima faciae, combinatorial indecomposability is stronger since
additive indecomposability corresponds to the special case where the
parts $B_0,\dots,B_{k-1}$ are required to be non-overlapping (possibly
empty) intervals. However, the additively indecomposable ordinals and
the combinatorially indecomposable ordinals are precisely the ordinal
powers of $\omega$, so the two properties are actually equivalent.

The fact that $\omega$ is combinatorially indecomposable is also known
as the Infinite Pigeonhole Principle. In reverse mathematics, the
Infinite Pigeonhole Principle was first studied by
Hirst~\cite{Hirst:Thesis}, who showed that it was equivalent to the
$\Pi^0_1$-Bounding Principle ($\Bnd{\Pi^0_1}$). The additive
indecomposability of the ordinal powers $\omega^\alpha$ was also
studied by Hirst~\cite{Hirst:OrdinalExponentiation}, but the formally
stronger combinatorial indecomposability of $\omega^\alpha$ was not
directly explored.

In this paper, we analyze the combinatorial indecomposability of
$\omega^n$ for $2 \leq n < \omega$. One difficulty with the analysis
is that ``$B_i$ has order-type $\omega^n$'' has several different
interpretations in second-order arithmetic. In Section~\ref{S:Indec},
we analyze the reverse mathematics of four natural intepretations
which are all equivalent assuming Arithmetic Comprehension ($\ACA_0$)
but diverge assuming only Recursive Comprehension ($\RCA_0$).

The analysis of the case $n = 2$ has led us to two combinatorial
principles related to Ramsey's Theorem for pairs ($\RT^2_k$), which has been
intensely studied in reverse
mathematics \cite{SeetapunSlaman,CholakJockuschSlaman,HirschfeldtShore,DzhafarovHirst,DzhafarovJockusch}.
\begin{axiom}{$\RT^2_k$}
  For every finite coloring $c: \NN^2 \to \set{0,\dots,k-1}$, there
  are a color $d < k$ and an infinite set $H$ such that $c(x,y) = d$
  for all $x, y \in H$ with $x < y$.
\end{axiom}
These two weaker principles are the Weak Ramsey Theorem for pairs
\begin{axiom}{$\wk\RT^2_k$}
  For every finite coloring $c: \NN^2 \to \set{0,\dots,k-1}$, there
  are a color $d < k$ and an infinite set $H$ such that $\set{ y \in
    \NN : c(x,y) = d}$ is infinite for every $x \in H$.
\end{axiom}
and the Hyper-Weak Ramsey Theorem for pairs
\begin{axiom}{$\hw\RT^2_k$}
  For every finite coloring $c:\NN^2\to\set{0,\dots,k-1}$, there are a
  color $d < k$ and an increasing function $h:\NN\to\NN$ such that,
  for all $0 < i_1 < i_2$, the rectangle
  \begin{equation*}
    [h(i_1-1),h(i_1)-1]\times [h(i_2-1),h(i_2)-1]
  \end{equation*}
  contains a pair with color $d$.
\end{axiom}
In Section~\ref{S:Forcing} we compare $\hw\RT^{2}_2$ to other known
combinatorial principles. In particular, we show that $\hw\RT^2_2$ is
strictly weaker than $\wk\RT^2_2$. In addition, we give a direct proof
that $\RCA_0 + \Ind\Sigma^0_2 + \hw\RT^{2}_2$ is
$\Pi^1_1$-conservative over $\RCA_0 + \Ind\Sigma^0_2$.

\begin{comment}
  Things to mention here\ldots

  Conventions:
  \begin{itemize}
  \item $\NN$ (internal numbers) vs $\omega$ (external numbers).
  \item Notation for standard parameters to ``theorem schemes.''
  \end{itemize}

\end{comment}

\paragraph{Conventions.}
A standard reference for subsystems of second-order arithmetic and
their use in reverse mathematics is Simpson~\cite{Simpson}. Formal
definitions of the basic systems $\RCA_0$ and $\ACA_0$ can be found
there. Another standard reference for induction principles used in
this paper is H{\'a}jek--Pudl{\'a}k~\cite{HajekPudlak}. While this
last reference focuses on first-order arithmetic, it is generally
straightforward to relativize their definitions and results to the
second-order setting.

Our general approach is model-theoretic rather than
proof-theoretic. Throughout the paper $\NN$ will denote the
first-order part of the model currently under consideration; we will
use $\omega$ to denote the set of standard natural numbers.
Every result in this paper indicates in parentheses the base system
over which the result is formulated. Some of the results are
parametrized by a standard natural number, which is also indicated in
parentheses.

In Section~\ref{S:Forcing}, for the purpose of forcing, we will find
it convenient to use a functional interpretation of the basic system
$\RCA_0$. Such a system was described by Kohlenbach~\cite{Kohlenbach},
but we will prefer the equivalent system described by Dorais~\cite{Dorais}. Our basic structures are of the form $\MN =
(\NN,\FN_1,\FN_2,\dots)$ where each $\FN_k$ is a set of functions
$\NN^k\to\NN$ which together form an algebraic clone: each $\FN_k$
contains all the constant functions, the projections
$\pi_i(x_1,\dots,x_k) = x_i$, and if $f \in \FN_\ell$ and
$g_1,\dots,g_\ell \in \FN_k$ then the superposition
$f(g_1(x_1,\dots,x_k),\dots,g_\ell(x_1,\dots,x_k))$ belongs to
$\FN_k$.

On top of this basic structure, we require closure under
\emph{primitive recursion}: there are distinguished $0 \in \NN$ (zero)
and $\sigma \in \FN_1$ (successor) such that for any $f \in \FN_{k-1}$
and $g \in \FN_{k+1}$ there is a unique $h \in \FN_k$ such
that \[h(0,\bar{w}) = f(\bar{w}) \quad\mbox{and}\quad
h(\sigma(x),\bar{w}) = g(h(x,\bar{w}),x,\bar{w})\] for all $x, \bar{w}
\in \NN$. Note that the uniqueness requirement on $h$ is crucial since
this is the only form of induction in this system. 

Using primitive recursion, we can define the usual arithmetic
operations such as addition, multiplication, truncated subtraction
(\(x \mathop{\dot{\smash{-}}} y = \max(x-y,0)\)) together with the
usual identities between them.  We will also assume the
\emph{dichotomy axiom} $x \mathop{\dot{\smash{-}}} y = 0 \lor y
\mathop{\dot{\smash{-}}} x = 0$, which is necessary to show that the
relation $x \leq y$ defined by $x \mathop{\dot{\smash{-}}} y = 0$ is a
linear ordering of $\NN$.

Finally, in addition to the basic axioms described above, we will
consider the second-order \emph{uniformization axiom}: For every $f
\in \FN_{k+1}$ such that $\forall \bar{w}\,\exists x\,{f(x,\bar{w}) =
  0}$, there is a $g \in \FN_k$ such that
$\forall\bar{w}\,{f(g(\bar{w}),\bar{w}) = 0}$. This axiom ensures
closure under general recursion, which is essentially equivalent to recursive comprehension.

Every functional structure $\MN$ corresponds to a set-based structure $(\NN;\mathcal{S};0,1,{+},{\cdot})$ for second-order arithmetic as described in~\cite{Simpson}, where $\mathcal{S}$ consists of all subsets of $\NN$ whose characteristic function is in $\FN_1$. The latter structure is a model of $\RCA_0$ if and only if the uniformization axiom holds in $\MN$. Conversely, given a traditional model $(\NN;\mathcal{S};0,1,{+},{\cdot})$ of $\RCA_0$, we can define $\FN_k$ to be the class of all functions $\NN^k\to\NN$ whose coded graph belongs to $\mathcal{S}$ and the resulting structure is a functional model which satisfies uniformization. Since our choice to adopt functional models is a matter of convenience, we will freely use this translation between functional models and traditional models.

\section{Combinatorial Indecomposability}\label{S:Indec}

In this section, we describe four different interpretations of the
statement that ``$\omega^n$ is combinatorially indecomposable'' and
examine their strength over $\RCA_0$.  We will state the
indecomposability principles in terms of a canonical representation of
the ordinal $\omega^n$, namely the lexicographic ordering of $\NN^n$,
which is defined by letting $(x_0,\ldots,x_{n-1}) <
(y_0,\ldots,y_{n-1})$ when
\begin{equation*}
  x_0 = y_0 \land \cdots\land x_{i-1} = y_{i-1} \land x_i < y_i
\end{equation*} 
holds for some $i < n$.  We also use the term lexicographic to
describe functions from $f:\NN^n\to\NN^n$ which preserve the lexicographic
ordering.

Usually, ``$X$ has order-type $\omega^n$'' is interpreted as saying
that the given ordering $X$ is order-isomorphic to $\omega^n$. Using
this interpretation, combinatorial indecomposability corresponds to
the following principle.

\begin{axiom}{$\XIndec_k^n$}
  For every finite coloring $c:\NN^n\to\set{0,\dots,k-1}$ there is a
  color $d < k$ such that the set
  \begin{equation*}
    A_d = \set{(x_1,\dots,x_n) \in \NN^n: c(x_1,\dots,x_n) = d}
  \end{equation*}
  is lexicographically isomorphic to $\NN^n$.
\end{axiom}

\noindent
We use $\XIndec^n$ to denote $(\forall k)\XIndec_k^n$.
Since a set $A \subseteq \NN$ is order-isomorphic to $\NN$ if and only
if it is infinite, the statement $\XIndec^1$ is precisely equivalent
to $\Bnd{\Pi^0_1}$ by Hirst's result. However, the very next case
$\XIndec_2^2$ already implies arithmetic comprehension.

\begin{proposition}[$\RCA_0$]\label{P:isom}
  $\XIndec_2^2$ implies arithmetic comprehension.
\end{proposition}

\begin{proof}
  We show that $\XIndec_2^2$ implies that the range of an
  arbitrary injection $f:\NN\to\NN$ exists. Consider the coloring
  $c:\NN^2\to 2$ defined by letting $c(x,y)=1$ if and only if $x \in
  \set{f(0),\ldots,f(y-1)}$. Note that $c(x,0) = 0$ for every $x$.

  Let $A_0 = c^{-1}(0)$ and $A_1 = c^{-1}(1)$. On the one hand, if
  $h:\NN^2\to A_1$ is an isomorphism, then $h_1(n,0)$ must be the
  $(n+1)$-th element of the range of $f$. On the other hand, if
  $h:\NN^2\to A_0$ is an isomorphism, then $h_1(n+1,0)-1$ must be the
  $(n+1)$-th element in the complement of the range of $f$.
\end{proof}

\noindent
Of course, it is easy to see that $\ACA_0$ proves $\XIndec^n$ for
all $n < \omega$.

\subsection{Indecomposability and Induction}

The weakest statements of indecomposability for $\omega^n$ that we will
consider are the following $\Pi^1_1$ statements.

\begin{axiom}{$\EIndec_k^n$}
  For every finite coloring $c: \NN^n \to \set{0,\dots,k-1}$ there is
  a color $d < k$ such that
  \begin{equation*}
    (\exists^\infty x_1)(\exists^\infty x_2)\cdots(\exists^\infty x_n)
    [c(x_1,x_2,\dots,x_n) = d].
  \end{equation*}
\end{axiom}

\noindent
We will use $\EIndec^n$ to denote $(\forall k)\EIndec_k^n$. Note that
$\EIndec_k^n$ is provable in $\RCA_0$ for every $k < \omega$, but the principle
$\EIndec^n$ is nontrivial.

The statement $\EIndec^1$ says that for every
finite coloring $c:\NN\to\set{0,\dots,k-1}$ there is a color $d < k$
such that the set $A_d = \set{x: c(x) = d}$ is infinite --- this
statement is equivalent to $\Bnd{\Pi^0_1}$.  We can generalize this as
follows.

\begin{theorem}[$\RCA_0$; $1 \leq n < \omega$]\label{T:elem}\mbox{}
  \begin{enumerate}[\bfseries\upshape(a)]
  \item\label{T:elem:ei2bnd} $\EIndec^n$ implies $\Bnd\Pi^0_{n}$.
  \item\label{T:elem:ind2ei} $\Ind\Sigma^0_{n+1}$ implies $\EIndec^n$.
  \end{enumerate}
\end{theorem}

\noindent
Part \eqref{T:elem:ei2bnd} of Theorem~\ref{T:elem} 
will follow from Proposition~\ref{P:Sapprox}.
Part \eqref{T:elem:ind2ei} is proved in Proposition~\ref{P:IndAndEIndec}.  

A principle equivalent to bounding will be used in the proof of Theorem~\ref{T:elem}.
In \cite[\S{I.2(b)}]{HajekPudlak}, H\'ajek and Pudl\'ak introduced
the Regularity Principle $\Reg{\Gamma}$ which says that if $\phi(x,y)$ is a $\Gamma$ formula then
\begin{equation*}
  (\exists^\infty x)(\exists y < k)\phi(x,y) \liff (\exists y < k)(\exists^\infty x)\phi(x,y)
\end{equation*}
holds for all $k \in \NN$. They further show that
$\Reg{\Sigma^0_{n+1}}$, $\Reg{\Pi^0_{n}}$, and $\Bnd{\Pi^0_{n+1}}$ 
are equivalent for every $n < \omega$ \cite[\S{I}.2.23(4)]{HajekPudlak}.

\begin{comment}

\begin{lemma}[$\RCA_0$]\label{L:JockuschSoare}
  If $A'$ exists and $f:\NN\to\{0,1\}$ is a partial $A$-computable
  function, then $f$ has a total extension $g$ such that $(g\oplus
  A)'\equiv A'$.
\end{lemma}

\begin{lemma}[$\RCA_0$]\label{L:Friedberg}
  Suppose $A'$ exists.  For any function $X:\NN\to\NN$ there is a
  function $B:\NN\to\NN$ such that $B' \equiv_T (B \oplus A)'
  \equiv_T (X \oplus A)'$.
\end{lemma}

\begin{proof}
  If possible, let $u(\tau,n)$ be the first $\sigma \supseteq \tau$
  such that $\Phi^{A,\sigma}_n(n){\downarrow}$; otherwise let
  $u(\tau,n) = \tau$. Note that $u$ is $A'$-computable and hence
  exists in our model.  Define by primitive recursion $\sigma_0 =
  \seq{}$ and
  \begin{equation*}
    \sigma_{n+1} = u(\sigma_n\cat\seq{X(n)},n).
  \end{equation*}
  Then, let $B=\bigcup_{n=0}^\infty \sigma_n$.

  It is easy to see that $(B\oplus A)'=(X\oplus A)'$ and that $B'\geq
  (B\oplus A)'$, so it remains to show that $B'$ can be computed from
  $(B\oplus A)'$.  First note that we can compute $X$ from $(B\oplus
  A)'$.  We can therefore compute the function which takes $n$ to
  $\sigma_n$.  Finally, notice that $e\in B'$ if and only if
  $\Phi_e^{A,\sigma_e}(e){\downarrow}$.  For if
  $\Phi_e^{A,\sigma_e}(e){\uparrow}$ then there is no extension
  $\sigma$ of $\tau_e$ such that $\Phi_e^{A,\sigma}(e){\downarrow}$.
\end{proof}

\end{comment}

The Regularity Priciple is useful in handling a certain class of colorings.
A function $c:\NN^{m+n}\to\NN$ is \emph{weakly $n$-stable} if for all
$x_1,\dots,x_m \in \NN$ there is a $y \in \NN$ such that
\begin{equation*}
  (\forall^\infty z_1)\cdots(\forall^\infty z_n)[y = c(x_1,\dots,x_m,z_1,\dots,z_n)].
\end{equation*}
This is very similar to saying that the iterated limit
\begin{equation*}
  \lim_{z_1\to\infty} \cdots \lim_{z_n\to\infty} c(x_1,\dots,x_m,z_1,\dots,z_n)
\end{equation*}
exists for all $x_1,\dots,x_m \in \NN$. However, the usual
definition of such limits requires that intermediate limits all exist
too, which is not required by weak $n$-stability. We say that $c$ is
\emph{strongly $n$-stable} if it is weakly $i$-stable for each $1
\leq i \leq n$; this guarantees the existence of all intermediate
limits and corresponds to the usual meaning of iterated limit.  The
two notions agree when $n = 1$ and they agree with the definition of
\emph{stable} introduced by Cholak, Jockusch, and
Slaman~\cite{CholakJockuschSlaman}.
 
If $c:\NN^{m+n}\to\NN$ is strongly $n$-stable then the iterated limit
\begin{equation*}
  f(x_1,\dots,x_m) = \lim_{z_1\to\infty} \cdots \lim_{z_n\to\infty} c(x_1,\dots,x_m,z_1,\dots,z_n)
\end{equation*}
is a total $\Sigma^0_{n+1}$-definable map $f:\NN^m\to\NN$.\footnote{More
  precisely, the graph of $f$ is $\Sigma^0_{n+1}$
  definable. Note that the map $f$ need not be a function of the
  current model.} The converse of this fact is is due to
\v{S}vejdar~\cite{Svejdar}.

\begin{lemma}[$\RCA_0 + \Bnd{\Pi^0_{n-1}}$; $1 \leq n < \omega$]\label{L:Limit}
  Every total $\Sigma^0_{n+1}$-definable map $f:\NN\to\NN$ is
  representable in the form
  \begin{equation*}
    f(x) = \lim_{z_1\to\infty} \cdots \lim_{z_n\to\infty} c(x,z_1,\dots,z_n),
  \end{equation*}
  where $c:\NN^{n+1}\to\NN$ is a strongly $n$-stable function.
\end{lemma}

\begin{proof}
  \v{S}vejdar \cite[Theorem~1]{Svejdar} shows under
  $\Bnd{\Pi^0_{n-1}}$ that that for every total
  $\Sigma^0_{n+1}$-definable map $f:\NN\to\NN$ there is a $1$-stable
  $\Sigma^0_n$-definable (indeed,
  $\Sigma^0_0(\Sigma^0_{n-1})$-definable) map $f':\NN^{2}\to\NN$ such
  that
  \begin{equation*}
    f(x) = \lim_{z_1\to\infty} f'(x,z_1)
  \end{equation*}
  for all $x \in \NN$. Iterating this result, we find
  $\Sigma^0_{n+1-i}$-definable strongly $i$-stable maps
  $f^{(i)}:\NN^{i+1}\to\NN$ such that
  \begin{equation*}
    f(x) = \lim_{z_1\to\infty}\cdots\lim_{z_i\to\infty} f^{(i)}(x,z_1,\dots,z_i)
  \end{equation*}
  for all $x \in \NN$. The $n$-th such map is
  $\Sigma^0_1$-definable and hence corresponds to an actual function
  $c:\NN^{n+1}\to\NN$ in our model which acts as claimed.
\end{proof}

\begin{proposition}[$\RCA_0$; $1 \leq n < \omega$]\label{P:Sapprox}
  $\Bnd{\Pi^0_n}$ is equivalent to the following statement.
  For any strongly $(n-1)$-stable $c:\NN^n\to\set{0,1,\dots,k-1}$, there is a $d < k$ such that
  \begin{equation*}
    (\exists^\infty x_1)(\exists^\infty x_2)\cdots(\exists^\infty x_n)[c(x_1,x_2,\dots,x_n) = d].
  \end{equation*}
\end{proposition}

\begin{proof}
  We will prove equivalence with $\Reg{\Pi^0_{n-1}}$ 
  instead of equivalence with $\Bnd{\Pi^0_n}$.

  Let $g:\NN\to\NN$ be the total $\Sigma^0_n$-definable function given by 
  \begin{equation*}
    g(x_1) = \lim_{x_2\to\infty}\cdots\lim_{x_n\to\infty} c(x_1,x_2,\dots,x_n)
  \end{equation*}
  Therefore, by $\Reg{\Sigma^0_n}$, we have that there is a $d<k$ such
  that $(\exists^\infty x_1)[g(x)=d]$. In particular,
  \begin{equation*}
    (\exists^\infty x_1)(\forall^\infty x_2)\cdots(\forall^\infty x_n)[c(x_1,x_2,\ldots,x_n)=d]
  \end{equation*}
  and the conclusion follows immediately.

  Let $\phi(x,y)$ be $\Pi^0_{n-1}$ and suppose that $(\exists^\infty
  x)(\exists y<k)\phi(x,y)$. Consider the total $\Sigma^0_n$-definable
  function $g$ such that $g(x_0) = y_0$ if and only if there is an $x$
  such that $x_0 \leq x$, $y_0 < k$ and $\phi(x,y_0) \land (\forall y
  < y_0)\lnot\phi(x,y)$ but $\lnot\phi(x',y')$ for all $x',y'$ such
  that $x_0 \leq x' < x$ and $y' < k$.  Since $g$ is a total
  $\Sigma^0_n$-definable function, Lemma~\ref{L:Limit} ensures that
  there is a strongly $(n-1)$-stable $c:\NN^n\to\set{0,\dots,k-1}$
  such that
  \begin{equation*}
    g(x_1) = \lim_{x_2\to\infty} \cdots \lim_{x_n\to\infty} c(x_1,x_2,\ldots,x_n)
  \end{equation*}
  for all $x_1$. By hypothesis, there is a $d<k$ such that
  \begin{equation*}
    (\exists^\infty x_1)(\exists^\infty x_2)\cdots(\exists^\infty x_n)[c(x_1,x_2,\ldots, x_n)=d].
  \end{equation*}
  It follows that $(\exists^\infty x)[g(x)=d]$ and hence that
  $(\exists^\infty x)\phi(x,d).$
\end{proof}

Note that part \eqref{T:elem:ei2bnd} of Theorem~\ref{T:elem} 
will follows immediately from Proposition~\ref{P:Sapprox}.
Now we prove part \eqref{T:elem:ind2ei}.
We will need the following result which is essentially due to Jockusch
and Stephan \cite{JockuschStephan}.

\begin{lemma}[$\RCA_0$]\label{L:JockuschStephan}
  Given a sequence of sets $A = \seq{A_n}_{n=0}^\infty$ such that
  $A''$ exists, there is an infinite set $X$ such that $(X \oplus A)''
  \equiv_T A''$ and, for all $n$, either $X \subseteq^* A_n$ or
  $X \subseteq^* \NN \setminus A_n$.
\end{lemma}

\noindent
Here and elsewhere, the notation $X \subseteq^* Y$ means that
$(\forall^\infty x)(x \in X \lthen x \in Y)$.  A close inspection of
the proof of~\cite[Theorem~2.1]{JockuschStephan} shows that the above
is provable in $\RCA_0$.

\begin{proposition}[$\RCA_0$; $1 \leq n < \omega$]\label{P:IndAndEIndec}
  $\Ind{\Sigma^0_{n+1}}$ implies $\EIndec^n$.
\end{proposition}

\begin{proof}
  Let $\MN$ be a model of $\RCA_0 + \Ind{\Sigma^0_{n+1}}$ and let
  $c_0:\NN^n\to\set{0,1,\dots,k-1}$ be a coloring in $\MN$.  Let $\MM$
  be the model of $\RCA_0$ whose second-order part consists of all
  $\Delta^0_{n+1}$-definable sets with parameters from $\MN$.

  Given $\bar{x}\in\NN^n$ and $i<k$, let
  $A_{\bar{x},i}=\set{y\in\NN:c_0(\bar{x},y)=i}$ and $A =
  \seq{A_n}_{n=0}^\infty$ effectively enumerate all such
  $A_{\bar{x},i}$.  Since $A'' \equiv_T c_0'' \in\MM$, by
  Lemma~\ref{L:JockuschStephan} there is an infinite set $X_1$ such
  that $(c_0 \oplus X_1)'' \equiv_T c_0''$ and, for all $\bar{x}$ and
  $i$, either $X_1\subseteq^* A_{\bar{x},i}$ or $X_1\subseteq^*
  \NN\setminus A_{\bar{x},i}$. We now define a new coloring
  $c_1:\NN^{n-1}\to\{0,1,\ldots, k-1\}$ by
  \begin{equation*}
    c_1(z_1,z_2,\ldots,z_{n-1}) = \lim_{x \in X_1} c_0(z_1,z_2,\ldots,z_{n-1},x),
  \end{equation*}
  which is computable from $(c_0 \oplus X_1)'$. Note also that $c_1'
  \leq_T c_0''$.

  If $n \geq 3$, we now repeat this process for the coloring $c_1$.
  For this construction to work, use the fact that $c_1'' \leq_T (c_0
  \oplus X_1)''' \equiv _T c_0''' \in\MM$ in order to apply
  Lemma~\ref{L:JockuschStephan} as above.  We are left with an
  infinite set $X_2$ such that $(c_1 \oplus X_2)'' \equiv_T c_1''
  \leq_T c_0'''$ and which defines a coloring
  \begin{equation*}
    c_2(z_1,\ldots,z_{n-2}) = \lim_{x \in X_2} c_1(z_1,\ldots,z_{n-2},x),
  \end{equation*}
  which is computable in $(c_1 \oplus X_2)'$.

  Continuing this process as necessary we end with a set $X_{n-1}$
  such that $(c_{n-2} \oplus X_{n-1})'' \equiv_T c_{n-2}'' \in \MM$
  and
  \begin{equation*}
    c_{n-1}(z_1) = \lim_{x\in X_{n-1}} c_{n-2}(z_1,x)
  \end{equation*}
  exists for all $z_1$.  Since $c_{n-1}' \leq_T c_{n-2}'' \leq_T
  c_0^{(n)} \in \MM$, there is a $d$ for which there are infinitely
  many $z$ such that $c_1(z) = d$.  Unraveling the definition of all
  the colorings we see that
  \begin{equation*}
    (\exists^\infty x_1)\ldots(\exists^\infty x_n)[c_0(x_1,\dots,x_n)=d]
  \end{equation*} 
  holds in $\MM$.  Therefore the same holds in $\MN$ since this is an
  arithmetical statement with parameters in $\MN$.
\end{proof}

\subsection{Indecomposability and Embeddings}

We now consider an indecomposability principle between
$\EIndec^n$ and $\XIndec^n$.
Much of the strength of $\XIndec^n$ comes from the isomorphism
requirement. This can be relaxed by asking instead that one of the
pieces of the partition \emph{contains} a lexicographically isomorphic
copy of $\NN^n$. Indeed, this is generally how combinatorial
indecomposability is understood for non-ordinal order
types \cite{Fraisse:Relations}. This leads us to our next formulation
of combinatorial indecomposability.

\begin{axiom}{$\LIndec_k^n$}
  For every finite coloring $c: \NN^n \to \set{0,\dots,k-1}$, there is
  a lexicographic embedding $h: \NN^n \to \NN^n$ such that $c \circ h$
  is constant.
\end{axiom}

\noindent
We will use $\LIndec^n$ to denote $(\forall k)\LIndec_k^n$. Again, we
see that $\LIndec^1$ is equivalent to $\Bnd{\Pi^0_1}$. The main result
of this section is that $\RCA_0 + \LIndec_2^3$ is equivalent to $\ACA_0$. 
Additionally, we show that $\LIndec_k^n$ implies $\EIndec_k^n$.
Note that $\LIndec_k^2$ is weaker than $\XIndec_k^2$,
since it follows from Ramsey's Theorem for pairs, which is known to 
be weaker than $\ACA_0$~\cite{SeetapunSlaman,CholakJockuschSlaman}.

To begin our analysis of $\LIndec^n$, we will first establish three
facts about the behavior of lexicographic embeddings in $\RCA_0$.
Except when explicitly stated otherwise, we will write $h_i$ for the
$i$-th coordinate of a lexicographic embedding $h:\NN^n\to\NN^n$.

\begin{lemma}[$\RCA_0$; $1 \leq n < \omega$]\label{L:lexfst}
  If $h:\NN^n\to\NN^n$ is a lexicographic embedding then
  \begin{equation*}
    x_1 \leq h_1(x_1,x_2,\dots,x_n) < h_1(x_1+1,0,\dots,0)
  \end{equation*}
  for all $x_1,\dots,x_n \in \NN$.
\end{lemma}

\begin{proof}
  By (external) induction on $1 \leq n < \omega$.  The case $n = 1$ is
  trivial.

  Suppose the result is true for some $n$.  Work in $\RCA_0$.  Let
  $h:\NN^{n+1}\to\NN^{n+1}$ be a lexicographic embedding.  For
  convenience, we will index our coordinates for $\NN^{n+1}$ from $0$
  to $n$ instead of $1$ to $n+1$.  Thus $h_0: \NN^{n+1}\to\NN$ is the
  first coordinate of $h$.

  We show that
  \begin{equation*}
    h_0(x_0,x_1,\dots,x_n) < h_0(x_0+1,0,\dots,0)
  \end{equation*}
  for all $x_0,x_1,\dots,x_n \in \NN$; the fact that $x_0 \leq
  h_0(x_0,x_1,\dots,x_n)$ then follows by induction.  Suppose, for the
  sake of contradiction, that $h_0(x_0,x_1,\dots,x_n) =
  h_0(x_0+1,0,\dots,0) = y_0$, say. Then the function
  $\tilde{h}:\NN^n\to\NN^n$ such that
  \begin{equation*}
    \tilde{h}_i(z_1,\dots,z_n) = h_i(x_0,x_1+1+z_1,z_2,\dots,z_n)
  \end{equation*}
  is a lexicographic embedding.  By the induction hypothesis,
  \begin{equation*}
    z_1 \leq \tilde{h}_1(z_1,0,\dots,0) = h_1(x_0,x_1+1+z_1,0,\dots,0) \leq h_1(x_0+1,0,\dots,0)
  \end{equation*}
  for all $z_1 \in \NN$, which is clearly impossible.
\end{proof}

\begin{lemma}[$\RCA_0$; $1 \leq n < \omega$]\label{L:lexlim}
  If $h:\NN^n\to\NN^n$ is a lexicographic embedding and $1 \leq j < i
  \leq n$, then
  \begin{equation*}
    \lim_{x_i\to\infty} h_j(x_1,\dots,x_{i-1},x_i,0,\dots,0)
  \end{equation*}
  exists and is bounded above by $h_j(x_1,\dots,x_{i-1}+1,0,\dots,0)$.
\end{lemma}

\begin{proof}
  We proceed by induction on $j < i$. By the induction hypothesis,
  find $\tilde{x}_i$ such that
  \begin{equation*}
    h_k(x_1,\dots,x_{i-1},x_i,0,\dots,0) = h_k(x_1,\dots,x_{i-1},\tilde{x}_i,0,\dots,0)
  \end{equation*}
  for all $x_i \geq \tilde{x}_i$ and $1 \leq k < j$.  Note that we
  must then have
  \begin{multline*}
    h_j(x_1,\dots,x_{i-1},x_i,0,\dots,0) \\
    \leq h_j(x_1,\dots,x_{i-1},x'_i,0,\dots,0) \\
    \leq h_j(x_1,\dots,x_{i-1}+1,0,0,\dots,0)
  \end{multline*}
  for all $x'_i \geq x_i \geq \tilde{x}_i$. It follows immediately
  that
  \begin{equation*}
    \lim_{x_i\to\infty} h_j(x_1,\dots,x_{i-1},x_i,0,\dots,0)
  \end{equation*}
  exists and is bounded above by
  $h_j(x_1,\dots,x_{i-1}+1,0,0,\dots,0)$.
\end{proof}

\begin{lemma}[$\RCA_0$; $1 \leq n < \omega$]\label{L:lexinf}
  If $h:\NN^n\to\NN^n$ is a lexicographic embedding and $1 \leq i \leq
  n$, then
  \begin{equation*}
    \lim_{x_i\to\infty} h_i(x_1,\dots,x_{i-1},x_i,0,\dots,0) = \infty
  \end{equation*}
  for all $x_1,\dots,x_{i-1} \in \NN$.
\end{lemma}

\begin{proof}
  By Lemma~\ref{L:lexlim}, we can find $\tilde{x}_i$ such that
  \begin{equation*}
    h_j(x_1,\dots,x_{i-1},x_i,0,\dots,0) = h_j(x_1,\dots,x_{i-1},\tilde{x}_i,0,\dots,0)
  \end{equation*}
  for all $x_i \geq \tilde{x}_i$ and all $1 \leq j < i$. Note that the
  function $\tilde{h}:\NN^{n-i+1}\to\NN^{n-i+1}$ defined by
  \begin{equation*}
    \tilde{h}_k(y_1,\dots,y_{n-i+1}) = h_{i+k-1}(x_1,\dots,x_{i-1},\tilde{x}_i+y_1,y_2,\dots,y_{n-i+1})
  \end{equation*}
  is then a lexicographic embedding and the result follows immediately
  by applying Lemma~\ref{L:lexfst} to $\tilde{h}$.
\end{proof}

\begin{theorem}[$\RCA_0$]\label{T:Lindec3ACA}
  $\LIndec_2^3$ implies arithmetic comprehension.
\end{theorem}

\begin{proof}
  We show how to compute the range of a function $f:\NN\to\NN$ using
  $\LIndec_2^3$. For each $z$, let $f[z] =
  \set{f(0),\dots,f(z)}$. Consider the coloring $c:\NN^3\to\set{0,1}$
  defined by
  \begin{equation*}
    c(x,y,z) = 
    \begin{cases}
      0 & \text{when $(\forall w \leq x)(w \in f[y] \liff w \in f[z])$,} \\
      1 & \text{otherwise.}
    \end{cases}
  \end{equation*}
  Suppose $h:\NN^3\to\NN^3$ is a lexicographic embedding such that $c
  \circ h$ is constant. First, note that $c \circ h$ must have
  constant value $0$.

  \begin{comment}
    Suppose instead that $c \circ h$ has constant value $1$. First
    find $y_0$ such that $h_1(0,y,0) = h_1(0,y_0,0)$ for all $y \geq
    y_0$. We can then find $z_0$ such that $h_2(0,y_0,z_0) <
    h_2(0,y_0+1,0) \leq h_3(0,y_0,z_0)$. Similarly, we can find can
    find pairs $(y_1,z_1),\dots,(y_k,z_k)$ such that
    $h_3(0,y_{i-1},z_{i-1}) \leq h_2(0,y_i,z_i) < h_3(0,y_i,z_i)$ for
    $1 \leq i \leq k$. Since $c(h(0,y_i,z_i)) = 1$, we can find $w_i
    \in f[h_3(0,y_i,z_i)] - f[h_2(0,y_i,z_i)]$ with $w_i \leq
    h_1(0,y_i,z_i) = h_1(0,y_0,0)$. Our choice of $y_i,z_i$ guarantees
    that the $w_i$ must be distinct, which is impossible when $k >
    h_1(0,y_0,0)$.
  \end{comment}

  To determine whether $x$ is in the range of $f$, use the following
  procedure:
  \begin{itemize}
  \item[] First find $y$ such that $h_1(x,y,0) = h_1(x,y+1,0)$. Answer
    yes if $x \in f[h_2(x,y+1,0)]$, otherwise answer no.
  \end{itemize}
  This procedure will never return false positive answers, so suppose
  that $x = f(s)$ and we check that the algorithm answers yes on input
  $x$.  The existence of a $y$ such that $h_1(x,y,0) = h_1(x,y+1,0)$
  is guaranteed by Lemma~\ref{L:lexlim}. Given such a $y$ we can then
  use Lemma~\ref{L:lexinf} to find $z$ such that $s \leq
  h_3(x,y,z)$. Since
  \begin{equation*}
    h_1(x,y,0) = h_1(x,y,z) = h_1(x,y+1,0), 
  \end{equation*}
  we then have
  \begin{equation*}
    h_2(x,y,0) \leq h_2(x,y,z) \leq h_2(x,y+1,0).
  \end{equation*}
  Since $c(h(x,y,z)) = 0$ and $x \leq h_1(x,y,z)$ by
  Lemma~\ref{L:lexfst}, we know that $x \in f[h_2(x,y,z)] \liff x \in
  f[h_3(x,y,z)]$. Since $s \leq h_3(x,y,z)$ we know that $x \in
  f[h_3(x,y,z)]$, and since $h_2(x,y,z) \leq h_2(x,y+1,0)$ we conclude
  that $x \in f[h_2(x,y+1,0)]$.
\end{proof}

We end this section by proving that $\LIndec_k^n$ implies $\EIndec_k^n$.
In light of Theorem~\ref{T:Lindec3ACA} and Theorem~\ref{T:elem},
this is really only interesting in the case $n=2$.

\begin{proposition}[$\RCA_0$; $1 \leq n < \omega$]
  For every positive integer $k$, $\LIndec_k^n$ implies $\EIndec_k^n$.
\end{proposition}

\begin{proof}
  By (external) induction on $n$, we show that {\itshape for any
    coloring $c: \NN^n \to \set{0,\dots,k-1}$, if there is a
    lexicographic embedding $h: \NN^n \to \NN^n$ such that $c \circ h$
    is constant with value $d < k$ then
    \begin{equation*}
      (\exists^\infty x_1)\cdots(\exists^\infty x_n)[c(x_1,\dots,x_n) = d].
    \end{equation*}}

  The result is trivial for $n = 1$. Suppose the result is true for
  some $n$. Work in $\RCA_0$. Let $c:\NN^{n+1}\to\set{0,\dots,k-1}$ be
  a coloring and let $h: \NN^{n+1} \to \NN^{n+1}$ lexicographic
  embedding $h:\NN^{n+1}\to\NN^{n+1}$ such that $c \circ h$ is
  constant with value $d < k$. For convenience, we will index our
  coordinates for $\NN^{n+1}$ from $0$ to $n$ instead of $1$ to
  $n+1$. Thus $h_0: \NN^{n+1}\to\NN$ is the first coordinate of $h$.

  Let $w_0 \in \NN$ be given, we want to show that
  \begin{equation}\tag{$\dagger$}\label{E:T1}
    (\exists x_0 \geq w_0)(\exists^\infty x_1)\cdots(\exists^\infty x_n)[c(x_0,x_1,\dots,x_n) = d].
  \end{equation}
  By Lemma~\ref{L:lexfst}, we have $w_0 \leq h_0(w_0,w_1,\dots,w_n) <
  h_0(w_0+1,0,\dots,0)$ for all $w_1,\dots,w_n \in \NN$.  By
  $\Ind\Sigma^0_1$, let
  \begin{align*}
    x_0 &= \max\set{h_0(w_0,w_1,\dots,w_n) : w_1,\dots,w_n \in \NN} \\
    &= \max\set{h_0(w_0,w_1,0,\dots,0) : w_1 \in \NN}
  \end{align*}
  and pick $w_1$ such that $x_0 = h_0(w_0,w_1,0,\dots,0)$.
  
  Define the coloring $c':\NN^n\to\set{0,\dots,k-1}$ by
  \begin{equation*}
    c'(x_1,\dots,x_n) = c(x_0,x_1,\dots,x_n),
  \end{equation*}
  and define the function $h':\NN^n\to\NN^n$ by
  \begin{equation*}
    h'_i(z_1,\dots,z_n) = h_i(w_0,w_1+z_1,z_2,\dots,z_n).
  \end{equation*}
  Then $h'$ is a lexicographic embedding is such that $c' \circ h'$ is
  constant with value $d$.  By the induction hypothesis applied to
  $h'$ and $c'$,
  \begin{equation*}
    (\exists^\infty x_1)\cdots(\exists^\infty x_n)[c'(x_1,\dots,x_n) = d].    
  \end{equation*}
  Since $x_0 \geq w_0$, this implies \eqref{E:T1}.
\end{proof}

\begin{corollary}[$\RCA_0$]
  $\LIndec^2$ implies $\Bnd{\Pi^0_2}$.
\end{corollary}

\begin{proof}
  $\EIndec^2$ implies $\Bnd{\Pi^0_2}$ by Theorem~\ref{T:elem}.
\end{proof}

\subsection{Indecomposability and Games}

Another formulation of combinatorial indecomposability is obtained by
interpreting the conclusion of $\EIndec_k^n$ in Hintikka's
Game-Theoretic Semantics. 
This process leads to the following game.

\begin{definition}
  Given a finite coloring $c: \NN^n\to\set{0,\dots,k-1}$, the game
  $\GINO_n(c)$ between \HUGO\ and \TINO\ is played as follows.
  \begin{itemize}
  \item To start the game, \HUGO\ chooses a color $d < k$.
  \item Then, \TINO\ and \HUGO\ alternately play
    \begin{equation*}
      \begin{array}{c|cccccccc}
        \text{\TINO} & a_1 & & a_2 & & \cdots & & a_n & \\
        \hline
        \text{\HUGO} & & b_1 & & b_2 & & \cdots & & b_n
      \end{array}
    \end{equation*}
    such that $a_i \leq b_i$ for $i = 1,\dots,n$.
  \end{itemize}
  \HUGO\ wins this play if $c(b_1,b_2,\dots,b_n) = d$, otherwise
  \TINO\ wins.
\end{definition}

\noindent
Of course, \TINO\ can never have a winning strategy for this game.

\begin{proposition}[$\RCA_0$; $1 \leq n < \omega$]
  For every finite coloring $c:\NN^n\to\set{0,\dots,k-1}$, \TINO\ does
  not have a winning strategy in the game $\GINO_n(c)$.
\end{proposition}
\begin{proof}
  Suppose that $(\sigma_d)_{d<k}$ is such that
  $\sigma_d:\NN^{<n}\to\NN$ is a winning strategy for \TINO\ in
  $\GINO_n(c)$ when \HUGO's first move is $d$. Define $b_1,\dots,b_n$
  by
  \begin{equation*}
    b_m = \max_{d<k} \sigma_d(b_1,\dots,b_{m-1})
  \end{equation*}
  for $m = 1,\dots,n$.  Then, for every $d \in \set{0,\dots,k-1}$,
  $b_1,b_2,\dots,b_n$ is a valid sequence of play for \HUGO\ against
  \TINO's strategy $\sigma_d$, which means that $c(b_1,b_2,\dots,b_n)
  \neq d$.  Therefore, $c(b_1,\dots,b_n) \notin \set{0,\dots,k-1}$ ---
  a contradiction.
\end{proof}

\noindent
If the game $\GINO_n(c)$ is determined, then \HUGO\ must have a
winning strategy, which leads to the following principle.

\begin{axiom}{$\GIndec_k^n$}
  For every finite coloring $c:\NN^n\to\set{0,\dots,k-1}$, \HUGO\ has
  a winning strategy in the game $\GINO_n(c)$.
\end{axiom}

\noindent
As usual, we use $\GIndec^n$ to denote $(\forall
k)\GIndec_k^n$. Again, it is easy to see that $\GIndec^1$ is
equivalent to $\Bnd{\Pi^0_1}$.

It turns out that $\GIndec_k^n$ is equivalent to a strong version of
$\LIndec_k^n$. A \emph{strong lexicographic embedding}
$h:\NN^n\to\NN^n$ is a lexicographic embedding with the additional
property that
\begin{equation*}
  x_1 = y_1,\dots,x_i=y_i \THEN h_i(x_1,\dots,x_n) = h_i(y_1,\dots,y_n)
\end{equation*}
holds for $i = 1,\dots,n$.
Characterizing $\GIndec^n$ as the existence of such strong 
lexicographic embedding relates $\GIndec^n$ to $\wk\RT^2_k$ and $\RT^2_k$.

\begin{proposition}[$\RCA_0$; $1 \leq n < \omega$]\label{P:strlexandgame}
  Given a finite coloring $c:\NN^n\to\set{0,\dots,k-1}$, \HUGO\ has a
  winning strategy in $\GINO_n(c)$ if and only if there is a strong
  lexicographic embedding $h:\NN^n\to\NN^n$ such that $c \circ h$ is
  constant.
\end{proposition}
\begin{proof}
  Suppose that $\sigma:\NN^{\leq n}\to\NN$ is a winning strategy for
  \HUGO. Let $d < k$, be \HUGO's color choice. For $i = 1,\dots, n$,
  define the function $h_i:\NN^i\to\NN$ by primitive recursion as follows
  \begin{align*}
    h_i(a_1,\dots,a_{i-1},0) &=
    \sigma(h_1(a_1),\dots,h_{i-1}(a_1,\dots,a_{i-1}),0), \\
    h_i(a_1,\dots,a_{i-1},a+1) &=
    \sigma(h_1(a_1),\dots,h_{i-1}(a_1,\dots,a_{i-1}),h_i(a_1,\dots,a_{i-1},a)+1).
  \end{align*}
  (The definition of $h_i$ depends on the prior definition of
  $h_1,\dots,h_{i-1}$, but since $n$ is standard this is not
  problematic.) Then the function
  \begin{equation*}
    h(a_1,\dots,a_n) = (h_1(a_1),\dots,h_n(a_1,\dots,a_n))
  \end{equation*}
  is a strong lexicographic embedding such that $c\circ h$ is constant
  with value $d$.
  
  Conversely, suppose that $h:\NN^n\to\NN^n$ is a strong lexicographic
  embedding such that $c \circ h$ is constant with value $d$. Define
  the strategy $\sigma:\NN^{\leq n}\to\NN$ as follows. Let $d$ be the
  initial color choice for $\sigma$ and then define
  $\sigma(a_1,\dots,a_i) = h_i(a_1,\dots,a_i,0,\dots,0)$. Then, by
  definition of strong lexicographic embedding, we always have
  \begin{equation*}
    h(a_1,\dots,a_n) = (\sigma(a_1),\dots,\sigma(a_1,\dots,a_n)),
  \end{equation*}
  which ensures that $\sigma$ is a winning strategy for \HUGO.
\end{proof}

\begin{corollary}[$\RCA_0$]\label{C:weakgame}
  $\GIndec_k^2$ is equivalent to $\wk\RT^2_k$.
\end{corollary}

\begin{corollary}[$\RCA_0$]
  $\RT^2_k$ implies $\GIndec_k^2$.
\end{corollary}

\begin{comment}
VERBAGE! We need to give some context for the next three results,
explaining the connection between $\Delta^0_n$ sets and games.
\end{comment}

\begin{proposition}[$\RCA_0$; $1 \leq n < \omega$]\label{P:limgame}
  Suppose $f:\NN^{1+n}\to\NN$ is a weakly $n$-stable function.  Then there
  is a coloring $c:\NN^{1+2n}\to\set{0,1}$ such that if \HUGO\ has a
  winning strategy in the game $\GINO_{1+2n}(c)$ then there are an
  infinite set $H$ and a function $f_\infty:H\to\NN$ such that
  \begin{equation*}
    (\forall^\infty z_1)\cdots(\forall^\infty z_n)[f_\infty(x) = f(x,z_1,\dots,z_n)]
  \end{equation*}
  holds for every $x \in H$.
\end{proposition}
\begin{proof}
  Let $c:\NN^{1+2n}\to\set{0,1}$ be defined by
  \begin{equation*}
    c(x,y_1,\dots,y_n,z_1,\dots,z_n) = 
    \begin{cases}
      1 & \text{when $f(x,\bar{y}) = f(x,\bar{z})$,} \\
      0 & \text{otherwise.}
    \end{cases}
  \end{equation*}

  Suppose $\sigma:\NN^{\leq1+2n}\to\NN$ is a winning strategy for
  \HUGO\ in $\GINO_{1+2n}(c)$. First note that since $f$ is weakly
  $n$-stable, the color $1$ must be \HUGO's first move.

  Now, knowing that \HUGO's first move is $1$, let
  \begin{equation*}
    H = \set{\sigma(w): w \in \NN} = \set{x \in \NN: (\exists w \leq
      x)[x = \sigma(w)]}.
  \end{equation*}
  This is clearly an infinite set. For $x \in H$, define $f_\infty(x)$
  as follows: let $w \leq x$ be least such that $x = \sigma(w)$, then
  let
  \begin{align*}
    y_1 &= \sigma(w,0), & y_2 &= \sigma(w,0,0), & \dots&, & y_n =
    \sigma(w,0,\dots,0),
  \end{align*}
  finally set $f_\infty(x) = f(x,y_1,\dots,y_n)$. The remainder of
  \HUGO's strategy $\sigma$ witnesses that
  \begin{equation*}
    (\exists^\infty z_1)\cdots(\exists^\infty z_n)[f_\infty(x) = f(x,z_1,\dots,z_n)].
  \end{equation*}
  Since $f$ is $n$-stable, it follows that
  \begin{equation*}
    (\forall^\infty z_1)\cdots(\forall^\infty z_n)[f_\infty(x) = f(x,z_1,\dots,z_n)].
  \end{equation*}
  as required.
\end{proof}

\begin{corollary}[$\RCA_0+\Bnd\Pi^0_{2n-1}$; $1 \leq n < \omega$]\label{C:stablelimgame}
  Suppose $f:\NN^{1+n}\to\NN$ is a weakly $n$-stable function.  Then there
  is a coloring $c:\NN^{1+2n}\to\set{0,1}$ such that if \HUGO\ has a
  winning strategy in the game $\GINO_{1+2n}(c)$ then there is a function
  $f_\infty:\NN\to\NN$ such that
  \begin{equation*}
    (\forall^\infty z_1)\cdots(\forall^\infty z_n)[f_\infty(x) = f(x,z_1,\dots,z_n)]
  \end{equation*}
  for every $x$.
\end{corollary}
\begin{proof}
  The function $\bar{f}:\NN^{1+n}\to\NN$ defined by
  \begin{equation*}
    \bar{f}(x,z_1,\dots,z_n) = \seq{f(0,z_1,\dots,z_n),\dots,f(x,z_1,\dots,z_n)}.
  \end{equation*}
  is also weakly $n$-stable by $\Bnd\Pi^0_{2n-1}$; apply
  Proposition~\ref{P:limgame} to $\bar{f}$.
\end{proof}


Here is a partial converse of
Proposition~\ref{P:limgame} for strongly $n$-stable functions.

\begin{proposition}[$\RCA_0$; $1 \leq n < \omega$]\label{P:gamelim}
  If $f^{(i)}:\NN^{1+i}\to\NN$ are such that
  \begin{equation*}
    f^{(i-1)}(x,y_1,\dots,y_{i-1}) = \lim_{y_i\to\infty} f^{(i)}(x,y_1,\dots,y_{i-1},y_i)
  \end{equation*}
  for $i = 1,\dots,n$, then \HUGO\ has a winning strategy in the game
  $\GINO_{1+2n}(c)$, where $c:\NN^{1+2n}\to\set{0,1}$ is the coloring
  associated to $f^{(0)}$ as in Proposition~\ref{P:limgame}.
\end{proposition}

\begin{comment}
  When $f$ is merely weakly $n$-stable, all we can say is that
  $f_\infty$ is $\Delta^0_{2n}$-definable. I don't know if every
  $\Delta^0_{2n}$ function is the limit of some computable weakly
  $n$-stable function; this suggests that maybe
  yes. This would be of great interest in obtaining degree
  lower-bounds for Ramsey's Theorem.
\end{comment}

\noindent
\HUGO's strategy is to simply pick sufficiently large natural numbers
with the value prescribed by the functions $f^{(i)}$.

When $n=1$, Propositions~\ref{P:limgame} and~\ref{P:gamelim} are exact
converses for stable $f$. In general, these two propositions show that
every particular instance of $\Delta_{1+n}$-comprehension corresponds
to \HUGO\ having a winning strategy in a particular instance of the
game $\GINO_{1+2n}$.

\section{The Hyper-Weak Ramsey Theorem}\label{S:Forcing}

In the last section, we left open some of questions regarding the
various statements of indecomposability for $\omega^2$. Not too
surprisingly, these principles are closely related to Ramsey's Theorem
for pairs. We have shown in Corollary~\ref{C:weakgame} that
$\GIndec^2_k$ is equivalent to the principle $\wk\RT^2_k$ from the
introduction.  The other principle from the introduction, namely
$\hw\RT^2_k$, turns out to be a close relative of $\LIndec^2_k$. In
its general form, the hyper-weak Ramsey Theorem is as follows.

\begin{axiom}{$\hw\RT^n_k$}
  For every finite coloring $c:\NN^n\to\set{0,\dots,k-1}$, there are a
  color $d < k$ and an increasing function $h:\NN\to\NN$ such that,
  for all $0 < i_1 < \cdots < i_n$, the box
  \begin{equation*}
    [h(i_1-1),h(i_1)-1]\times\cdots\times [h(i_n-1),h(i_n)-1]
  \end{equation*}
  contains an $n$-tuple with color $d$.
\end{axiom}

\noindent
The main result of this section is the following.

\begin{theorem}\label{T:HWForcing}
  Every countable model of $\RCA_0 + \Ind\Sigma^0_{2}$ has an
  $\omega$-extension that satisfies $\RCA_0 + \Ind\Sigma^0_{2} +
  \hw\RT^{2}_2$.
\end{theorem}

\noindent
It follows immediately that $\RCA_0 + \Ind\Sigma^0_2 + \hw\RT^{2}_2$
is $\Pi^1_1$-conservative over $\RCA_0 + \Ind\Sigma^0_2$.

The principle $\hw\RT^2_k$ can be reformulated as follows.

\begin{proposition}[$\RCA_0$]\label{P:HWEquiv}
  The principle $\hw\RT^2_k$ is equivalent to the following statement.
  \begin{principle} 
    For every finite coloring $c:\NN^2\to\set{0,\dots,k-1}$, there are
    a color $d < k$ and an increasing function $g:\NN\to\NN$ such that
    \begin{equation}\label{E:HWEquiv:hom}
      \bigcup_{x=g(i-1)}^{g(i)-1} \set{y \in \NN: c(x,y) = d}
    \end{equation}
    is infinite for every $i \geq 1$.
  \end{principle}
\end{proposition}

\begin{proof}
  The fact that $\hw\RT^2_k$ implies this statement is clear. For the
  converse, let $d < k$ and $g:\NN\to\NN$ be as in the statement. For
  each $i \geq 1$ let $f_i(z)$ be the first $y \geq z$ such that
  $c(x,y) = d$ for some $g(i-1) \leq x \leq g(i)-1$. Define the
  sequence $0 < i_0 < i_1 < i_2 < \cdots$ so that $g(i_{\ell+1}) >
  f_{i_0}(g(i_\ell)),\dots,f_{i_\ell}(g(i_\ell))$ for each
  $\ell$. Then $h(\ell) = g(i_\ell)$ is as required.
\end{proof}

\noindent
If $c:\NN^2\to\set{0,\dots,k-1}$ is a coloring and $h:\NN^2\to\NN^2$
is a lexicographic embedding such that $c \circ h$ is constant with
value $d < k$, then the first coordinate function $g(i) = h_1(i,0)$ is
such that each of the sets~\eqref{E:HWEquiv:hom} is infinite.

\begin{corollary}[$\RCA_0$]
  $\LIndec_k^2$ implies $\hw\RT^2_k$.
\end{corollary}

The principle $\hw\RT^2_2$ is also related to the principle
$\ADS$ of Hirschfeldt and Shore~\cite{HirschfeldtShore}. A coloring
$c:[\NN]^2\to\set{0,\ldots,k-1}$ is \emph{transitive} if, for all $x <
y < z$, if $c(x,y) = c(y,z)$ then $c(x,z) = c(x,y) = c(y,z)$.

\begin{axiom}{$\ADS$}
  Every transitive coloring $c:[\NN]^2\to\set{0,1}$ has an infinite
  homogeneous set.
\end{axiom}

\noindent
For every transitive coloring $c:[\NN]^2\to\set{0,1}$ there is a
unique linear ordering ${\prec}$ such that, for all $x < y$,
\begin{equation*}
  c(x,y) = 
  \begin{cases}
    0 & \text{when $x \succ y$,} \\
    1 & \text{when $x \prec y$.}
  \end{cases}
\end{equation*}
Thus $\ADS$ is equivalent to the statement that every linear ordering
of $\NN$ has an infinite ascending or descending sequence, hence the
name. The principle $\st\ADS$ is the restriction of $\ADS$ to stable
transitive colorings.

\begin{proposition}[$\RCA_0$]
  $\hw\RT^2_2$ implies $\st\ADS$.
\end{proposition}

\begin{proof}
  Let $c:[\NN]^2\to\set{0,1}$ be a stable transitive coloring. By
  $\hw\RT^2_2$, there is are a color an increasing function
  $h:\NN\to\NN$ such that, for all $0 < i < j$, the
  rectangle \[[h(i-1),h(i)-1]\times[h(j-1),h(j)-1]\] contains a pair
  with color $d$. Let ${\prec}$ be the linear ordering of $\NN$ which
  agrees with color $d$, and for each $i$ let $m(i)$ be the
  ${\prec}$-minimal element of the interval $[h(i-1),h(i)-1]$. Note
  that $m(i)$ is necessarily ${\prec}$-below all but finitely elements
  of $\NN$. Therefore, we can define the sequence $i_0 = 0 < i_1 < i_2
  < \cdots$ so that each $i_{k+1}$ is the least $i > i_k$ such that
  $c(m(i_k),m(i)) = d$. By transitivity, the sequence
  $\set{m(i_k)}_{k=0}^\infty$ is an infinite $c$-homogeneous set.
\end{proof}

\noindent
It was shown by Chong, Lempp, and Yang~\cite{ChongLemppYang} that
$\st\ADS$ implies $\Bnd\Pi^0_1$.

\begin{corollary}[$\RCA_0$]
  $\hw\RT^2_2$ implies $\Bnd\Pi^0_1$.
\end{corollary}

\begin{comment}
  Below, I will also implicitly consider the following aymmetric
  strengthening of $\hw\RT^2_2$.

  \begin{axiom}{$\hw\RT^2_{2+}$}
    For every finite coloring $c:\NN^2\to\set{0,1}$, there is an
    increasing function $h:\NN\to\NN$ such that either
    \begin{equation*}
      \bigcup_{x=h(i-1)}^{h(i)-1} \set{y \in \NN: c(x,y) = 1}
    \end{equation*}
    is infinite for every $i \geq 1$, or else
    \begin{equation*}
      \bigcup_{x=h(i-1)}^{h(i)-1} \set{y \in \NN: c(x,y) = 0}
    \end{equation*}
    is cofinite for every $i \geq 1$.
  \end{axiom}

  \noindent
  I don't think this is actually stronger than $\hw\RT^2_2$, but I
  don't have a proof.
\end{comment}

\subsection{Forcing Preliminaries}

The forcings we will be interested in are forcings with finite
conditions. That is, our poset $\QQ$ of forcing conditions is a
$\Sigma^0_1$-definable set and so are the order relation ${\leq}$ and
the incompatibility relation ${\perp}$. We will first develop the
general theory of such forcings before we deal with actual examples to
prove Theorem~\ref{T:HWForcing}.

Our approach to forcing follows that of Dorais~\cite{Dorais}. In particular, we work within the functional
interpretation of $\RCA_0$ presented in the introduction. For the rest
of this section we show how to adapt the forcing machinery
of~\cite{Dorais} to forcings with finite conditions.

We first develop the basic machinery necessary to define the internal
forcing language. The base level of this are the forcing names, which
are the terms of the forcing language. A \emph{partial $k$-ary name}
is a $\Sigma^0_1$-definable set $F \subseteq \QQ\times\NN^{k+1}$ (with
ground model parameters) such that:
\begin{itemize}
\item If $(p,\bar{x},y) \in F$ and $q \leq p$ then $(q,\bar{x},y) \in
  F$.
\item If $(p,\bar{x},y), (p,\bar{x},y') \in F$ then $y = y'$.
\end{itemize}
We say that $F$ is a \emph{$p$-local} if for every $q \leq p$ and
every $\bar{x} \in \NN$ there are $y \in \NN$ and $r \leq q$ such that
$(r,\bar{x},y) \in F$.

Before we discuss the syntax of the forcing language, we will discuss
the semantics of these names. A filter $G \subseteq \QQ$ is
\emph{$\Pi^0_n$-generic} (over $\MN$) if for every set $D \subseteq
\QQ$ which is $\Pi^0_n$ definable over $\MN$, there is a $p \in G$
such that either $p \in D$ or else $p$ has no extension in $D$ at all.

If $G$ is $\Pi^0_1$-generic and $F$ is a $p$-local $k$-ary name for
some $p \in G$, then the \emph{evaluation} $F^G$ is the total $k$-ary
function defined by
\begin{equation*}
  F^G(\bar{x}) = y \quad\IFF\quad (\exists q \in G)((q,\bar{x},y) \in F).
\end{equation*}
The basic projections, constants, and indeed all ground model
functions $F$ have \emph{canonical names} $\check{F}$ defined by
\begin{equation*}
  (p,\bar{x},y) \in \check{F} \quad\IFF\quad y = F(\bar{x}),
\end{equation*} 
which invariably evaluate to $F$. The \emph{generic extension}
$\MN[G]$ is the $\omega$-extension of $\MN$ whose functions consist of
the evaluations of all names that are $p$-local for some $p \in G$.

In a typical language, the basic terms are composed to form the class
of all terms. This is not so for the forcing language since
composition and other operations can be done directly at the semantic
level. If $F$ is a partial $\ell$-ary name and $F_1,\dots,F_\ell$ and
are partial $k$-ary names then the superposition $H = F \circ
(F_1,\dots,F_\ell)$ is defined by \[(p,\bar{x},z) \in H \IFF \exists
\bar{y}\,((p,\bar{x},y_1) \in F_1 \land\cdots\land (p,\bar{x},y_\ell)
\in F_\ell \land (p,\bar{y},z) \in F).\] This is a partial $k$-ary
name and if each of $F, F_1,\dots,F_\ell$ is $p$-local, then so is
$H$. Primitive recursion can be handled in a similar way. Given
partial a $(k-1)$-ary name $F_0$ and a $(k+1)$-ary name $F$, we the
$k$-ary name $H$ is defined by $(p,\bar{x},y,z) \in H$ iff there is a
finite sequence $\seq{z_0,\dots,z_y}$ with $z = z_y$ such that
$(p,\bar{x},z_0) \in F_0$ and $(p,\bar{x},i,z_i,z_{i+1}) \in F$ for
every $i < y$. This is again a partial $k$-ary name and if $F_0, F$
are $p$-local then so is $H$. Other recursive operations can be
handled via Proposition~\ref{P:GenUnif}.

The \emph{formulas of the forcing language} are defined in the usual
manner as the smallest family which is closed under the following
formation rules.
\begin{itemize}
\item If $F$ is a partial $k$-ary name, $ F'$ is a partial $k'$-ary
  name, and $\bar{v} = v_1,\dots,v_k$, $\bar{v}' = v_1',\dots,v_{k'}'$
  are variable symbols then $(F(\bar{v}) = F'(\bar{v}'))$ is a
  formula.
\item If $\phi$ is a formula then so is $\lnot\phi$.
\item If $\phi$ and $\psi$ are formulas then so is $(\phi\land\psi)$.
\item If $\phi$ is a formula and $v$ is a variable symbol, then
  $(\forall v)\psi$ is also a formula.
\end{itemize}
Free and bound variables are defined in the usual manner. The
\emph{sentences} of the forcing language are formulas without free
variables.  Although not present in the formal language, we will
freely use $\lor$, $\lthen$, $\liff$ and $\exists$ as abbreviations:
\[\begin{array}{c@{\qquad}c}
  ( \phi \lor \psi) \equiv \lnot(\lnot\phi\land\lnot\psi), &
  (\phi \lthen \psi) \equiv \lnot(\phi\land\lnot\psi), \\
  (\phi \liff \psi) \equiv (\phi\lthen\psi)\land(\psi\lthen\phi), &
  (\exists v)\phi \equiv \lnot(\forall v)\lnot\phi.
\end{array}\]
Because our language lacks ${\leq}$, bounded quantifiers are defined
by \[(\forall v \leq F)\phi \equiv (\forall v)(v
\mathop{\dot{\smash{-}}} F = 0 \lthen \phi), \quad (\exists v \leq
F)\phi \equiv (\exists v)(v \mathop{\dot{\smash{-}}} F = 0 \land
\phi).\]
Bounded formulas are those whose quantifiers are all of this form,
where the name $F$ does not depend on the quantified variable $v$.
The usual arithmetic hierarchy is then built from these in the usual
manner by alternation of quantifiers.

We are now ready to define the forcing relation. The definition for
atomic sentences is motivated by the above definition of the forcing
extension. The remaining cases follow the classical definition of
forcing. The definition of $p \forces \theta$ is by induction on the
complexity of the sentence $\theta$. Assume all names that occur in
sentences below are $p$-local.
\begin{itemize}
\item $p \forces (F = F')$ iff, for all $q \leq p$ and $y, y' \in
  \NN$, if $(q,y) \in F$ and $(q,y') \in F'$ then $y = y'$.
\item $p \forces (\phi\land\psi)$ iff $p \forces \phi$ and $p \forces
  \psi$.
\item $p \forces (\forall v)\phi(v)$ iff $p\forces \phi(\check{x})$,
  for all $x \in \NN$.
\item $p \forces \lnot\phi$ iff there is no $q \leq p$ such that $q
  \forces \phi$.
\end{itemize}
The meaning of the forcing relation for the abbreviations defined
above can be computed as usual.
\begin{itemize}
\item $p \forces (F \neq F')$ iff, for all $q \leq p$ and $y,y' \in
  \NN$, if $(q,y) \in F$ and $(q,y') \in F'$ then $y \neq y'$.
\item $p \forces (\phi\lor\psi)$ iff for every $q \leq p$ there is a
  $r \leq q$ such that either $r \forces \phi$ or $r \forces \psi$.
\item $p \forces (\phi\lthen\psi)$ iff for every $q \leq p$ such that
  $q \forces \phi$, there is a $r \leq q$ such that $r \forces \psi$.
\item $p \forces (\exists v)\phi(v)$ iff for every $q \leq p$ there
  are a $r \leq q$ and a $x \in \NN$ such that $r \forces
  \phi(\check{x})$.
\end{itemize}
Moreover, this is a classical forcing:
\begin{itemize}
\item $p \forces \lnot\lnot\phi$ iff $p \forces \phi$.
\end{itemize}

\begin{lemma}\label{L:BoundedForcing}
  For every bounded formula $\phi(\bar{v})$ of the forcing language,
  there is a partial name $T_\phi(\bar{v})$ such that $\phi(\bar{v})$
  is $p$-local if and only if $T_\phi(\bar{v})$ is $p$-local, and then
  $p \forces (\forall\bar{v})[\phi(\bar{v}) \liff T_\phi(\bar{v}) =
  0].$
\end{lemma}

\begin{proof}
  We define $T_\phi(\bar{v})$ by induction on the complexity of $\phi(\bar{v})$.
  \begin{itemize} 
  \item $T_{F=F'}(\bar{v})=|F(\bar{v})-F'(\bar{v})|$.
  \item $T_{\neg\phi}(\bar{v})=1 \mathop{\dot{\smash{-}}} T_\phi(\bar{v})$.
  \item $T_{\phi\land\psi}(\bar{v})=T_\varphi(\bar{v})+T_\psi(\bar{v})$.
  \item $T_{(\forall w\leq F)\phi}(\bar{v}) =\sum_{w\leq F(\bar{v})} T_{\phi}(\bar{v},w)$.
  \end{itemize}
  The fact that $T_\phi(\bar{v})$ is $p$-local when $\phi$ is $p$-local
  follows from the fact that $p$-local names are closed under
  superposition and primitive recursion.
\end{proof}

The following fact is then easy to check by induction on $n$.

\begin{proposition}[$\RCA_0$; $1 \leq n < \omega$]\label{P:DefForcing}
  If $\theta(v_1,\dots,v_k)$ is a $p$-local $\Pi^0_n$-formula of the
  forcing language then the relation $p \forces
  \theta(\check{x}_1,\dots,\check{x}_k)$ is $\Pi^0_n$, uniformly in
  the parameter $p$.
\end{proposition}

\begin{comment}
  By Lemma~\ref{L:BoundedForcing}, every $p$-local $\Pi^0_1$-formula
  $\theta(\bar{v})$ is equivalent to one of the form
  $(\forall\bar{w})(T(\bar{v},\bar{w}) = 0)$ where $T$ is a $p$-local
  name. The relation $p \forces \theta(\check{\bar{x}})$ is then
  equivalent to the $\Pi^0_1$ formula
  \begin{equation*}
    (\forall\bar{y},z)(\forall q \leq p)((q,\bar{x},\bar{y},z) \in T
    \lthen z = 0).
  \end{equation*}
  This handles the base case.

  Suppose $\theta(\bar{v}) \equiv
  (\forall\bar{w})\lnot\phi(\bar{v},\bar{w}),$ where
  $\phi(\bar{v},\bar{w})$ is a $\Pi^0_n$-formula of the forcing
  language.  Then $p \forces \theta(\check{\bar{x}})$ is equivalent to
  \begin{equation*}
    (\forall\bar{y})(\forall q \leq p)(q \nforces \phi(\check{\bar{x}},\check{\bar{y}})),
  \end{equation*}
  which is equivalent to a $\Pi^0_{n+1}$ formula by the induction hypothesis.
\end{comment}

\noindent
Note, however, that the $\Sigma^0_n$ forcing relation is not generally
$\Sigma^0_n$.

It follows that if $G$ is a $\Pi^0_n$-generic filter over $\MN$, then
for every $\Pi^0_n$-sentence $\theta$ which is $p$-local for some $p
\in G$, there is a condition $q \in G$ such that either $q \forces \theta$
or $q \forces \lnot\theta$. Working through the inductive definitions,
we see that in this scenario
\begin{equation*}
  \MN[G] \models \theta^G 
  \quad\IFF\quad 
  \mbox{$q \forces \theta$ for some $q \in G$} 
\end{equation*}
where $\theta^G$ is the standard formula obtained by replacing all
names of $\theta$ by their evaluations.

\begin{proposition}[$\RCA_0$]\label{P:GenUnif}
  If $\theta(\bar{v},w)$ is a $p$-local $\Sigma^0_1$ formula of the forcing
  language such that $p \forces (\forall\bar{v})(\exists w)\theta(\bar{v},w)$
  then there is a $p$-local name $F$ such that $p \forces
  (\forall\bar{v})\theta(\bar{v},F(\bar{v}))$.
\end{proposition}

\begin{proof}
  We may assume that
  $\theta(\bar{v},w)$ is actually bounded and moreover of the form $T(\bar{v},w) = 0$ for some
  $p$-local name $T$ as in Lemma~\ref{L:BoundedForcing}. Let $(q_n,\bar{x}_n,y_n)$ enumerate the
  $\Sigma^0_1$-definable set 
  \begin{equation*}
    \set{(q,\bar{x},y) \in \QQ \times \NN^{k+1} : q \leq p \land (q,\bar{x},y,0) \in T}
  \end{equation*}
  Then define the name $F$ by $(q_n,\bar{x}_n,y_n) \in F$ iff
  for every $m < n$, we have $q_m \perp q_n$, or $\bar{x}_m \neq
  \bar{x}_n$, or $y_m = y_n$.
\end{proof}

If $G$ is a filter over $\QQ$ then let $\MN[G]$ be the model obtained
by evaluating all $G$-local names at $G$. Since elements of $\MN$ all
have canonical names, we see that $\MN[G]$ is an $\omega$-extension of
$\MN$.

\begin{theorem}[$\RCA_0$]\label{T:PresRCA}
  If $G$ is $\Pi^0_2$-generic for $\QQ$ over $\MN$, then $\MN[G]
  \models \RCA_0$.
\end{theorem}

\noindent
This is a consequence of Proposition~\ref{P:GenUnif} and
the fact that $G$-local names are closed under superposition and
primitve recursion.

Note that Theorem~\ref{T:PresRCA} can fail if the assumption on $G$ is
weakened to $\Pi^0_1$-genericity. This is because we only use names
which are $p$-local for some $p \in G$, which is not always sufficient
to guarantee closure under recursive comprehension. However, if
$\MN[G] \models \Ind\Sigma^0_1$, then $\MN[G]$ can be closed under
recursive comprehension to form a model of $\RCA_0$.

\subsection{Forcing Construction}

Let $\PP = (P,{\leq})$ be the poset of all (codes for) finite increasing functions
$p:\set{0,1,\dots,|p|-1}\to\NN$, ordered by end-extension (this is a
variant of Cohen forcing). Let $c:\NN^2\to\set{0,1}$ be a coloring in
$\MN$ for which there is no increasing $h:\NN\to\NN$ such that the set
\begin{equation}\label{E:hwhom1}
  \bigcup_{x=h(n)}^{h(n+1)-1} \set{ y \in \NN : c(x,y) = 1 }
\end{equation}
is infinite for every $n$. Let $\PP' = (P',{\leq'})$ be the subposet
consisting of all $p \in \PP$ such that
\begin{equation}\label{E:hwhom0}
  \bigcup_{x=p(n)}^{p(n+1)-1} \set{ y \in \NN : c(x,y) = 0 }
\end{equation}
is \emph{cofinite} for every $n < |p|-1$. The poset $\PP'$ is
$\Sigma^0_2$-definable over $\MN$, so our methods do not necessarily
apply for forcing with $\PP'$ over $\MN$. Instead, we force over a
\emph{$\Sigma^0_2$-envelope of $\MN$}: an $\omega$-extension $\MN'$ of
$\MN$ such that $\MN' \models \RCA_0$ and every total
$\Sigma^0_2$-definable function over $\MN$ belongs to $\MN'$. (Since
$\MN \models \Ind{\Sigma^0_2}$, the model $\MN'$ consisting of all
total functions which are $\Sigma^0_2$-definable over $\MN$ is as
required, but we will need the more general notion later.)

The hypothesis that there is no increasing function $h:\NN\to\NN$ such
that the sets~\eqref{E:hwhom1} are all infinite clearly implies that
each one of the sets $D_n' = \set{p \in \PP' : |p| \geq n}$ is open
dense. Since a generic filter $G$ for $\PP'$ must meet each one of
these open dense sets, we see that $g = \bigcup G$ is a well-defined
increasing function $g:\NN\to\NN$ such that~\eqref{E:hwhom0} is
cofinite for each $n$. This function $g$ is the \emph{generic real}
associated to $G$. The generic filter $G$ is in fact completely
determined by $g$ since $G = \set{p \in \PP : p \subseteq g}$. Since
$g$ will be of greater interest, we will systematically work with $g$
instead of $G$ throughout the following.

The following fact is the keystone to showing that forcing with $\PP'$
over $\MN'$ leads to a generic function $g$ which is well behaved over
the $\omega$-submodel $\MN$.

\begin{lemma}\label{L:S1Density}
  Suppose $U \subseteq \PP$ is $\Sigma^0_1$-definable over $\MN$. If
  every $p \in \PP'$ is such that for every $q \leq' p$ there is a
  $r \leq q$ such that $r \in U$, then for every $q \leq' p$ there is
  a $r \leq' q$ such that $r \in U$.
\end{lemma}

\begin{proof}
  Suppose that every $q \leq' p$ has an extension in $U$, but there is
  some $q \leq' p$ which has no extension in $U \cap \PP'$. We will
  use such a $q$ to construct an increasing function $h:\NN\to\NN$ in
  $\MN$ such that the set~\eqref{E:hwhom1} is infinite for every $i
  \geq 1$, thereby contradicting our hypothesis that there are no such
  functions.

  First, find $h(0) \in \NN$ such that $q_0 \in \PP'$ where $q_0 =
  q\cat{h(0)}$. By hypothesis, we can find an extension $r_0 \geq q_0$
  in $U$. It follows that $r_0 \notin \PP'$, which means that
  \begin{equation*}
    \bigcup_{x=r_0(i-1)}^{r_0(i)-1} \set{y \in \NN: c(x,y) = 0}
  \end{equation*}
  is coinfinite for some $0 < i < |r_0|$. Since $q_0 \in \PP'$, this
  $i$ must be greater than or equal to $|q_0|$. So if we set $h(1) =
  \max(r_0) = r_0(|r_0|-1)$, we necessarily have
  \begin{equation*}
    \bigcup_{x=h(0)}^{h(1)-1} \set{y \in \NN: c(x,y) = 1}
  \end{equation*}
  is infinite.

  Once $h(n)$ has been defined, set $q_n = q\cat{h(n)}$ and note that
  $q_n$ is necessarily in $\PP'$ since $h(n) \geq h(0)$. As above, we
  can then find an extension $r_n \leq q_n$ in $U$ and set $h(n+1) =
  \max(r_n) = r_n(|r_n|-1)$. As before, we then have that
  \begin{equation*}
    \bigcup_{x=h(n)}^{h(n+1)-1} \set{y \in \NN: c(x,y) = 1} 
  \end{equation*}
  is infinite.

  This construction can be carried out completely inside
  $\MN$. Indeed, all we need to do at each stage is to search for an
  extension $r_n \leq q_n$ in $U$, which can be done by enumerating
  $U$ until such $r_n$ is found.
\end{proof}

\begin{proposition}\label{P:Locality}
  If $p \in \PP'$ and $F$ is a $p$-local name for $\PP$ over $\MN$,
  then $F$ is also $p$-local for $\PP'$ over $\MN'$.
\end{proposition}

\begin{proof}
  Apply Lemma~\ref{L:S1Density} to the sets $U_{\bar{x}} = \set{r \in
    \PP: (\exists y)[(r,\bar{x},y) \in F]}$.
\end{proof}

\noindent
Note that there are names for $\PP$ over $\MN$ which are $p$-local for
$\PP'$ over $\MN'$, but not $p$-local for $\PP$ over $\MN$. One such
name is the $2$-ary name $f$ such that $f(n,m)$ is the $(m+1)$-th element
of
\begin{equation*}
  \bigcup_{x=g(n)}^{g(n+1)-1} \set{ y \in \NN : c(x,y) = 0 }.
\end{equation*}
In particular, the generic function $g$ is not $\Pi^0_2$-generic for
$\PP$ over $\MN$. It is however weakly $\Sigma^0_2$-generic for $\PP$
over $\MN$ as we will now show.

\begin{proposition}\label{P:Sigma2ForcingXfer}
  Let $\theta$ be a $p$-local $\Sigma^0_2$ sentence of the forcing
  language for $\PP$ over $\MN$. If $p \in \PP'$ and $p \forces'
  \theta$ then there is a $q \leq' p$ such that $q \forces \theta$.
\end{proposition}

\begin{proof}
  By Lemma~\ref{L:BoundedForcing}, we may suppose that $\theta$ is of
  the form $(\exists u)(\forall v)[F(u,v) = 0]$ where $F$ is a $2$-ary
  $p$-local name for $\PP$ over $\MN$. Without loss of generality, we
  may further assume that there is some $x \in \NN$ such that, for all
  $y \in \NN$, $p \forces' F(\check{x},\check{y}) = 0$. Applying
  Lemma~\ref{L:S1Density} to the set $U = \set{r \in \PP : (\exists
    y,z)[(r,x,y,z) \in F \land z \neq 0]}$, we see that there must be
  a $q \leq' p$ with no extension in $U$ at all. This is equivalent to
  saying that $q \forces F(\check{x},\check{y}) = 0$ for all $y \in
  \NN$. Hence, $q \forces \theta$.
\end{proof}

Together with the generic extension $\MN'[g]$ of $\MN'$, we obtain an
$\omega$-extension $\MN[g]$ of $\MN$ by evaluating
all partial names in $\MN$ which are $g$-local for $\PP'$ over
$\MN'$. This is not a generic extension, but it does satisfy
$\RCA_0$. In order to iterate the forcing construction, we will need
to make sure that the generic extension $\MN'[g]$ is a
$\Sigma^0_2$-envelope for $\MN[g]$. The key to prove this is the
following fact.

\begin{proposition}\label{P:Lowness}
  Let $\theta(u,v)$ be a $p$-local $\Sigma^0_2$ formula of the forcing
  language for $\PP$ over $\MN$. If $p \forces' (\forall u)(\exists
  w)\theta(u,w)$ then there is a $p$-local name $F$ for $\PP'$ over
  $\MN'$ such that $p \forces' (\forall u)\theta(u,F(u))$.
\end{proposition}

\begin{proof}
  Suppose $\theta(u,v) \equiv (\exists w)\phi(u,v,w)$, where
  $\phi(u,v,w)$ is a $p$-local $\Pi^0_1$ formula of the forcing
  language for $\PP$ over $\MN$. By Proposition~\ref{P:DefForcing},
  the relation
  \begin{equation*}
    R = \set{(q,x,y,z) \in \PP\times\NN^3:
      q \forces \phi(\check{x},\check{y},\check{z})}
  \end{equation*} 
  is $\Pi^0_1$ definable over $\MN$. Therefore $R' = R \cap
  \PP'\times\NN^3 \in \MN'$. Fix an enumeration
  $\seq{r_n,x_n,y_n,z_n}$ of $R'$ and define the partial name $F$ by
  $(q,x,y) \in F$ iff there is an $n$ such that $x = x_n$, $y = y_n$,
  and $q \leq r_n$ but $q \nleq r_m$ for $m <
  n$. Proposition~\ref{P:Sigma2ForcingXfer} shows that $F$ is a
  $p$-local name and that $p \forces' (\forall u)\theta(u,F(u))$, as required.
\end{proof}

\noindent
Thus, if the $\Sigma^0_2$ formula $\theta(u,v)$ defines a map over
$\MN[g]$, then this function actually belongs to the generic envelope
$\MN'[g]$. In other words, $\MN'[g]$ is a $\Sigma^0_2$-envelope for
$\MN[g]$.

\begin{proof}[Proof of Theorem~\ref{T:HWForcing}]
  We start with a model $\MN_0$ of $\RCA_0 + \Ind{\Sigma^0_2}$. Then
  we find a $\Sigma^0_2$-envelope $\MN_0'$ for $\MN_0$ as explained
  above. Let $c:\NN^2\to\set{0,1}$ be a coloring in $\MN_0$ for which
  there is no increasing $h:\NN\to\NN$ such that
  \begin{equation*}
    \bigcup_{x=h(n)}^{h(n+1)-1} \set{y : c(x,y) = 0}
  \end{equation*}
  is infinite for every $n$. Then we force with $\PP'$ over $\MN_0'$
  to obtain a generic extension $\MN_1' = \MN_0'[g]$ and at the same
  time an $\omega$-extension $\MN_1 = \MN_0[g]$, where $g:\NN\to\NN$
  is an increasing function such that
  \begin{equation*}
    \bigcup_{x=g(n)}^{g(n+1)-1} \set{y : c(x,y) = 0}
  \end{equation*}
  is cofinite for every $n$. By Proposition~\ref{P:Lowness}, we then
  have that $\MN_1'$ is a $\Sigma^0_2$-envelope for $\MN_1$, hence
  $\MN_1 \models \Ind{\Sigma^0_2}$.

  We can iterate this process to obtain two parallel sequences of
  $\omega$-extensions
  \begin{equation*}
    \begin{array}{ccccccccccc}
      \MN_0 & \subseteq & 
      \MN_1 & \subseteq &
      \MN_2 & \subseteq & 
      \cdots & \subseteq & 
      \MN_\omega &=& \bigcup_{i<\omega} \MN_i \\
      \cap && \cap && \cap && && \cap \\
      \MN_0' & \subseteq &
      \MN_1' & \subseteq &
      \MN_2' & \subseteq &
      \cdots & \subseteq & 
      \MN_\omega' &=& \bigcup_{i<\omega} \MN_i' \\
    \end{array}
  \end{equation*}
  At each stage, we have that $\MN_i$, $\MN_i'$ are both models of
  $\RCA_0$ and $\MN_i'$ is a $\Sigma^0_2$-envelope of $\MN_i$. It
  follows that these facts are also true for $\MN_\omega$ and
  $\MN_\omega'$. Therefore $\MN_\omega \models \RCA_0 +
  \Ind\Sigma^0_2$. Moreover, with careful bookkeeping to deal with
  every potential counterexample $c:\NN^2\to\set{0,1}$ of
  $\hw\RT^2_2$, we can make sure that $\MN_\omega \models
  \hw\RT^2_2$. In the end, $\MN_\omega$ is the required
  $\omega$-extension.
\end{proof}

\subsection{Forcing over $\omega$-models}

When $\MN$ is an $\omega$-model of $\RCA_0$, the forcing methods of
the last section are slight overkill. Indeed, there is no risk of
breaking induction by adjoining more second-order elements to
$\MN$. Nevertheless, the forcing posets used in the last section can
be used to shed some light on the situation for $\omega$-models.

\begin{proposition}[$\ACA_0$]\label{P:1GenLow}
  For every computable coloring $c:\NN^2\to\set{0,1}$ one of the following is
  true. 
  \begin{itemize}
  \item There is a computable increasing function $h:\NN\to\NN$ such
    that 
    \begin{equation*}
      \bigcup_{x=h(n)}^{h(n+1)-1} \set{ y \in \NN : c(x,y) = 1 }
    \end{equation*}
    is infinite for every $n$.
  \item There is a $0'$-computable $1$-generic increasing function
    $g:\NN\to\NN$ such that
   \begin{equation*}
      \bigcup_{x=g(n)}^{g(n+1)-1} \set{ y \in \NN : c(x,y) = 0 }
    \end{equation*}
    is cofinite for every $n$.
 \end{itemize}
\end{proposition}

\begin{proof}
  Suppose there is no computable increasing function $h:\NN\to\NN$
  such that
  \begin{equation*}
    \bigcup_{x=h(n)}^{h(n+1)-1} \set{ y \in \NN : c(x,y) = 1 }
  \end{equation*}
  is infinite for every $n$. Let $\PP$ and $\PP'$ be defined as in the
  previous section. Note that Lemma~\ref{L:S1Density} applies to the
  countable coded $\omega$-model $\mathrm{REC}$ whose second-order
  part consists of all computable sets, as computed in our ambient
  model of $\ACA_0$. In our ambient model of $\ACA_0$, we have an effective
  listing $\seq{U_n}_{n=0}^\infty$ of all computably enumerable
  ($\Sigma^0_1$ over $\mathrm{REC}$) subsets of $\PP$. 

  Since $\PP'$ is $\Sigma^0_2$-definable (without parameters) we have
  a $0'$-computable enumeration $\seq{p_i}_{i=0}^\infty$ of $\PP'$.
  Define the sequence $\seq{q_n}_{n=0}^\infty$ of elements of $\PP'$
  as follows.
  \begin{itemize}
  \item Let $q_0$ be an arbitrary element of $\PP'$.
  \item Once $q_n$ has been defined, let $q_{n+1}$ be the first $p_i$
    in our enumeration such that $p_i \leq' q_n$ and either $p_i \in U_n$ or else there is no extension $r
    \leq p_i$ such that $r \in U_n$.
  \end{itemize}
  Lemma~\ref{L:S1Density} ensures that there always is a $q_{n+1}
  \leq' q_n$ as required by the second condition. Furthermore, since
  each $U_n$ is computably enumerable, the requirements for the second
  condition can be checked using $0'$ as an oracle. 

  It follows that the sequence $\seq{q_n}_{n=0}^\infty$ is
  $0'$-computable and hence so is $g = \bigcup_{n=0}^\infty q_n$. Note
  that $g$ is a well-defined increasing function $\NN\to\NN$ since our
  listing $\seq{U_n}_{n=0}^\infty$ includes all of the open dense sets
  $\set{p \in \PP : |p| \geq i}$. Moreover, $g$ is clearly $1$-generic
  and since every initial segment of $g$ is in $\PP'$, we see that
 \begin{equation*}
    \bigcup_{x=g(n)}^{g(n+1)-1} \set{ y \in \NN : c(x,y) = 0 }
  \end{equation*}
  is cofinite for every $n$.
\end{proof}

\noindent
Since $1$-generic degrees below $0'$ are low, by iterating the
relativized form of Proposition~\ref{P:1GenLow}, we see that:

\begin{corollary}[$\ACA_0$]
  There is a countable coded $\omega$-model of $\RCA_0 + \hw\RT^2_2$ whose
  second-order part consists entirely of low sets.
\end{corollary}

\noindent
It was shown by Downey, Hirschfeldt, Lempp, and
Solomon~\cite{DowneyHirschfeldtLemppSolomon}
that $\st\RT^2_2$ has no $\omega$-model whose second-order part
consists entirely of low sets. A close inspection of their argument
shows that this can be formalized in $\ACA_0$.

\begin{corollary}[$\ACA_0$]
  $\hw\RT^2_2$ does not imply $\st\RT^2_2$ over $\RCA_0$.
\end{corollary}

\noindent
It follows that $\hw\RT^2_2$ also doesn't imply $\wk\RT^2_2$ since the
latter implies $\st\RT^2_2$ over $\RCA_0$.

\section{Conclusions and Questions}

In the first part of this paper, we investigated various formulations
of the pigeonhole principle for finite ordinal powers of the ordinal
$\omega$. The weakest such principle $\EIndec^n$ was found to be
sandwiched between two standard induction principles. Namely,
Theorem~\ref{T:elem} showed that 
\begin{equation*}
  \xymatrix{
    \Ind{\Sigma^0_{n+1}} \ar[r] & \EIndec^n \ar[r] & \Bnd{\Pi^0_n}.
  }
\end{equation*}
It is known that $\Ind{\Sigma^0_{n+1}}$ is strictly stronger than
$\Bnd{\Pi^0_n}$~\cite{HajekPudlak}, but we don't know how $\EIndec^n$
sits in between the two.

\begin{question}
  Does $\EIndec^n$ lie strictly between $\Ind{\Sigma^0_{n+1}}$ and
  $\Bnd{\Pi^0_n}$ in the hierarchy of induction principles?
\end{question}

\noindent
Hirst's result that $\EIndec^1$ is equivalent to $\Bnd{\Pi^0_1}$
suggests that $\EIndec^n$ might be equivalent to
$\Bnd{\Pi^0_n}$. Indeed, there is hope that some induction could be
shaved off from our proof of Proposition~\ref{P:IndAndEIndec}.

The two stronger principles $\LIndec^n_k$ and $\GIndec^n_k$ turned out
to be equivalent to $\ACA_0$ when $n \geq 3$ and $k \geq 2$. However,
for $n = 2$, both $\LIndec^2_k$ and $\GIndec^2_k$ follow from
$\RT^2_k$ (indeed $\wk\RT^2_k$), which is known to be strictly weaker
than $\ACA_0$~\cite{SeetapunSlaman,CholakJockuschSlaman}. This led us
to consider two weak forms of Ramsey's Theorem for pairs, namely
$\wk\RT^2_k$ and $\hw\RT^2_k$. Another interesting possible weakening
of $\RT^2_k$ was considered by Dzhafarov and
Hirst~\cite{DzhafarovHirst}, namely the Increasing Polarized Theorem
for pairs ($\IPT^2_k$), which is sandwitched between $\RT^2_k$ and
$\wk\RT^2_k$. The known implications between these principles in the
case $k = 2$ are summarized in the following diagram:
\begin{equation*}
  \xymatrix{
    \RT^2_2 \ar[r] & \IPT^2_2 \ar[r] & \GIndec_2^2 \ar@{<->}[r] \ar[d] & \wk\RT^2_2 \ar[r] \ar[d] & \st\RT^2_2 \ar[d] \\
    && \LIndec_2^2 \ar[r] & \hw\RT^2_2 \ar[r] \ar[ru]|| & \st\ADS
    \ar[r] & \Bnd\Pi^0_1.
  }
\end{equation*}
Besides for the non-implications $\hw\RT^2_2 \not\lthen \st\RT^2_2$,
$\Bnd{\Pi^0_1} \not\lthen \st\ADS$, and their consequences, we do not
know whether any of the remaining implications are strict. Many of the
resulting questions are special cases or refinements of the open
questions
from~\cite{CholakJockuschSlaman,HirschfeldtShore,DzhafarovHirst}. For
example, it is still an open question whether $\st\RT^2_2$ implies
$\RT^2_2$~\cite[Question~13.6]{CholakJockuschSlaman}. Of the remaining
questions, we wonder the following.

\begin{question}\label{Q:hwlindec}
  Is $\hw\RT^2_2$ strictly weaker than $\LIndec^2_2$?
\end{question}

\begin{question}\label{Q:wklindec}
  Is $\wk\RT^2_2$ strictly stronger than $\LIndec^2_2$?
\end{question}

\noindent
Of course, a negative answer to Question~\ref{Q:hwlindec} would
provide a positive answer to Question~\ref{Q:wklindec}. Similarly, a
negative answer to Question~\ref{Q:wklindec} would provide a positive
answer to Question~\ref{Q:hwlindec}. However, it is plausible that
both questions have a positive answer. 

In another line of thought, we wonder how $\hw\RT^2_2$ is related to
other combinatorial consequences of Ramsey's Theorem for pairs. Of
particular interest is the following.

\begin{question}
  How is $\hw\RT^2_2$ related to the principle $\st\CAC$ of
  Hirschfeldt and Shore~\cite{HirschfeldtShore}?
\end{question}

\noindent
Indeed, the similarity between the forcing construction from
Section~\ref{S:Forcing} and those used by Hirschfeldt and Shore
suggests that there might be some non-trivial ties between these two
principles.

\bibliographystyle{amsplain}
\bibliography{igame}

\end{document}